%% file: paper-arxiv.tex
\documentclass{amsart}
\usepackage[dvipsnames]{xcolor} %
\usepackage{amssymb,amsmath,amsfonts,amsthm, enumerate, tikz, float,enumitem, comment, mathtools, bm, xfrac, empheq, esvect, soul}
\usetikzlibrary{arrows.meta, positioning, fit, trees}
\usepackage[pdfencoding=auto, psdextra]{hyperref}
\mathtoolsset{showonlyrefs}
\usepackage{subcaption}

\usepackage{theoremref}
\newtheorem{theorem}{Theorem}

\newtheorem{lemma}[theorem]{Lemma}

\newtheorem{remark}[theorem]{Remark}

\newtheorem{claim}[theorem]{Claim}
 
\newcommand{\f}{\frac}

\renewcommand{\P}{\mathbf{P}}

\newcommand{\ind}[1]{\mathbf{1}{\{ #1 \}}}

\newcommand{\E}{\mathbf E}

\DeclareMathOperator{\Exp}{Exp}

\usepackage{lineno}

\newcommand{\ta}[1]{}%

\newcommand{\om}[1]{}%

\definecolor{darkgreen}{rgb}{0,0.5,0}

\usetikzlibrary[topaths]
\newcount\mycount

\title{Chase-escape with conversion as a multiple sclerosis lesion model}

\author[]{Emma Bailey} \address{University of Bristol} \email{e.c.bailey@bristol.ac.uk }

\author[]{Erin Beckman} \address{Utah State University} \email{erin.beckman@usu.edu }

\author[]{Sara\'{i} Hern\'{a}ndez-Torres} \address{Universidad Nacional Aut\'onoma de M\'exico} \email{saraiht@im.unam.mx}

\author[]{Matthew Junge} \address{Baruch College}  \email{Matthew.Junge@baruch.cuny.edu}

\author[]{Aanjaneya Kumar} \address{Sante Fe Institute} \email{aanjaneya@santafe.edu}

\author[]{Ann Lee} \address{Hunter College} \email{annleeny9@gmail.com}

\author[]{Danny Li} \address{Baruch College}  \email{danny.li2@baruchmail.cuny.edu}

\author[]{tahda queer} \address{Hunter College and Columbia University} \email{taqueer@proton.me}

\author[]{Alisher Raufov} \address{Baruch College} \email{alisherkhon.raufov@baruchmail.cuny.edu}

\author[]{Lily Reeves} \address{California Institute of Technology} \email{lreeves@caltech.edu}

\author[]{Omer Rondel} \address{Baruch College}  \email{omer.rondel@baruchmail.cuny.edu}

\begin{document}

\begin{abstract}
We introduce conversion to the stochastic process known as chase-escape in an effort to model aspects of inflammatory damage from multiple sclerosis. We prove monotonicity results for aggregate damage for the model on the positive integers, trees, stars, and the complete graph. Additionally, we establish the existence and asymptotic order of a phase transition on bounded degree graphs with a non-trivial site percolation threshold.
\end{abstract}

\thanks{Part of this research was conducted during a Mathematical Research Intensive supported by NSF DMS Grant 2238272 that also partially supports MJ, AL, DL, tq, AR, and OR. SHT acknowledges support from UNAM-PAPIIT grant IA103724. LR is partially supported by the NSF MSPRF Award DMS-2303316. Thanks to Josh Cruz for assistance with Figure 2.}
\maketitle

\section{Introduction}

Multiple sclerosis (MS) is a chronic disease characterized by lesions of damaged white matter in the central nervous system (CNS).  Roughly speaking, MS lesions are formed when inflammatory T-cells recruit macrophages and B-cells to attack myelin in the CNS. The process is eventually suppressed and halted by regulatory T-cells, often leaving behind a lesion of permanently damaged CNS tissue \cite{dobson2019multiple, lassmann2012progressive}.
A prominent issue with studying MS is that it is extremely difficult to obtain dynamical information from patients or even animal models \cite{kotelnikova2015signaling}. 

Mathematical models have the potential to provide unique insights \cite{weatherley2023could}. 
   In their survey article, Weatherly et al.\ describe various MS modeling attempts \cite{weatherley2023could}. Lombardo et al.\ introduced an ordinary differential equations (ODE) model for the interactions among macrophages, chemoattractants and destroyed oligodendrocytes \cite{lombardo2017demyelination}. Kotelnikova et al.\ introduced a different ODE model describing interactions among axons and macrophages \cite{kotelnikova2015signaling}. They showed that their model can be adjusted to match different disease courses in MS patients. In \cite{moise2021mathematical}, Moise and Avner developed a more comprehensive model of lesion formation. It involved a complex system of differential equations involving dozens of agents and a three-dimensional space variable. Their model exhibited quantitative agreement with clinical data that measured total lesion volume in MS patients \cite{comi2001european, lublin2013randomized}. Travaglini recently developed a similar reaction-diffusion equation based approach \cite{model2025relapsing}. Schoonheim et al.\ proposed a network-based approach to describe MS dynamics and postulated that disease impairment heightens with ``network collapse" i.e., when important connective regions accrue too much damage \cite{schoonheim2022network}. 
  
The aforementioned models are deterministic. However, lesion formation appears to be partly driven by local random interactions \cite{lombardo2017demyelination}. Stochastic spatial models for remyelination in the CNS were proposed in \cite{holmes2014interactions,vissatmodelling}. Kim et al.\ introduced a programmed cell death model on random networks \cite{kim2009damage} that was later applied to MS lesion formation \cite{thamattoor2012simulation, mathankumar2013autoimmune}. Their simulation results suggested that preemptively killing cells through a process called apoptosis can limit total damage.

We model the interplay between inflammatory T-cells and regulatory T-cells by generalizing a stochastic growth model known as chase-escape to include spontaneous conversion to account for the arrival of regulatory cells. This line of inquiry aligns with recent experimental gene therapies to treat MS by boosting the production of regulatory T-cells \cite{keeler2018}. 
The two questions we seek to address are: 
\begin{enumerate}
    \item Is CNS tissue damage monotone in the inflammatory and suppression rates?
    \item Does tissue damage exhibit a phase transition?
\end{enumerate} 
Question (1) is not apriori obvious, since rapid inflammation might trigger more vigorous suppression, resulting in less damage as seen in the model from \cite{thamattoor2012simulation}. 
The basic idea of Question (2) is characterizing conditions that allow lesions to reach macroscopic size.

Since we approach these questions with full rigor, our model is a dramatic oversimplification. Nonetheless, our inquiry lies at the theoretical forefront of stochastic growth models. Our primary aim is to advance theoretical tools that may one day be sophisticated enough to capture more salient features of MS.

\subsection{Model definition}

  In our model, inflammatory $r$-particles ``escape" to and damage healthy $w$-sites while being ``chased" and suppressed by regulatory $b$-particles. {We will refer to sites with $r$- or $b$-particles as red or blue, respectively and healthy $w$-sites as white.}
Spontaneous ``conversion" of red to blue sites represents the arrival of regulatory cells that halt inflammation. Thus, red sites represent cells with active inflammation, and blue sites represent damaged cells at which there is no longer inflammation. To distinguish between the two mechanisms by which a blue particle can occupy a site, either by blue spreading to a site through chasing a red particle or by a red particle converting to blue, we will sometimes call the former \emph{chase} or \emph{predation} and the latter  we will refer to as \emph{conversion}.

Formally, \emph{chase-escape with conversion} takes place on a locally finite graph $G$ in which vertices are in one of the three states $\{w,r,b\}$. 
Adjacent vertices in states $(r,w)$ transition to $(r,r)$ according to independent Poisson processes with rate $\lambda$. 
Each vertex in state $r$ transitions to state $b$ according to an independent Poisson process with rate $\alpha$. 
Adjacent $(b,r)$ vertices transition to $(b,b)$ according to independent rate $1$ Poisson processes. 
The standard initial configuration has the root vertex $x_0$ in state $r$ and all other vertices in state $w$. See Figure~\ref{fig:dynamics} for a visual summary of the dynamics  and Figure~\ref{fig:graphs} for some examples of the initial configuration on various graphs. Note that the process is well-defined on any locally finite graph since the memoryless property of Poisson processes allows us to update the always finite configuration of red and blue sites in a Markovian manner. 

\begin{figure}[h!]
    \centering
	\begin{tikzpicture}[scale = .6]

\begin{scope}[shift={(0,0)}]

	\draw [fill=red] (0,0) circle [radius=0.35];
	\draw[thick] (.5,0) -- (1.75,0);
	\draw (2.25,0) circle [radius=0.35];
	\draw[->, thick] (3,0) -- (4,0);
	\draw[fill = red] (4.75,0) circle [radius=0.35];
	\node at (3.4,.5) {{$\lambda$}};
	\draw[thick] (5.25,0) -- (6.25,0);
	\draw[fill = red] (6.75,0) circle [radius=0.35];
\end{scope}

\begin{scope}[shift={(0,-1.5)}]

	\draw [fill=blue] (0,0) circle [radius=0.35];
	\draw[thick] (.5,0) -- (1.75,0);
	\draw[fill =red] (2.25,0) circle [radius=0.35];
	\draw[->, thick] (3,0) -- (4,0);
	\draw[fill = blue] (4.75,0) circle [radius=0.35];
	\node at (3.4,.5) {$1$};
	\draw[thick] (5.25,0) -- (6.25,0);
	\draw[fill = blue] (6.75,0) circle [radius=0.35];

\end{scope}

\begin{scope}[shift={(0,-3)}]
	\draw [fill=red] (2.25,0) circle [radius=0.35];
		\draw[->, thick] (3,0) -- (4,0);
		\node at (3.4,.5) {{$\alpha$}};
	\draw[fill = blue] (4.75,0) circle [radius=0.35];
\end{scope}
\end{tikzpicture}
    \caption{Chase-escape with conversion dynamics.}
    \label{fig:dynamics}
\end{figure}
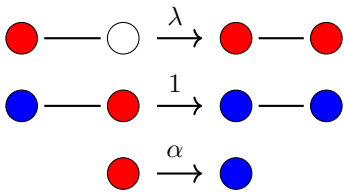

The special case $\alpha =0$ corresponds to the well-studied \emph{chase-escape} model introduced by ecologists Keeling, Rand, and Wilson \cite{keeling1995ecology, rand1995invasion} to study parasite-host relations. Further work has reinterpreted the dynamics as models for: predator-prey systems, rumor scotching, infection spread, and malware repair in a device network \cite{bordenave2014extinction, rumor, de2015process, gilbert, cande}. Note that for chase-escape, {both red and blue sites must be present in the initial configuration, and this is usually implemented by} adding an additional vertex in state $b$ attached to $x_0$. Also note that we could add a third parameter $\beta$ to modify the spreading rate of blue to obtain a larger class of models. For example, setting $\beta=0$ would yield the classical SIR model of infection spread \cite{durrett2010random}. We believe our results still hold for general $\beta>0$. However, since conversion is the main novelty, we opt for a simpler presentation with $\beta$ fixed at $1$.

For chase-escape with conversion on $G$ with site $x_0$ initially in state $r$, let $X=X(G,x_0,\lambda,\alpha)$ be the number of sites in state $b$ when the process eventually fixates. We adopt the convention that $X=\infty$ if the process never fixates. Much of the chase-escape literature has focused on the critical red spreading rate when $G$ is an infinite graph. The analogue for chase-escape with conversion is
\begin{align}
\lambda_c(\alpha) = \lambda_c(\alpha,G,x_0) \coloneqq \sup \{ \lambda \colon \P_{\lambda,\alpha}(X < \infty ) = 1 \}, \label{eq:lc}
\end{align}
above which $X$ is infinite with positive probability \cite{cande}. This definition could  be adapted to growing sequences of finite graphs as was done for chase-escape in \cite{bernstein2021chase}.

Another intriguing feature of chase-escape is the apparent difficulty to prove monotonicity in $\lambda$. Recall that a nonnegative random variable $X'$ is stochastically smaller than another nonnegative random variable $X$ (denoted $X' \preceq X$) if and only if there is a coupling such that $X' \leq X$ almost surely, or equivalently $\P(X' \geq  a) \leq \P(X \geq a)$ for all $a \geq 0$. When $G$ contains cycles, it remains an open problem in most natural settings (for example $\mathbb Z^2$ \cite{kumar2021chase}) to prove stochastic monotonicity of $X$ in $\lambda$.  Although faster spread of red particles is believed to cause red to reach more sites, there is the offsetting effect that more red sites means more opportunities for blue to spread. A similar effect happens with conversion. Faster conversion should reduce $X$,  but more red conversion means fewer red sites for blue to chase.

\subsection{Results}

We are interested in whether or not $X$ is monotone in $\lambda$ and $\alpha$.  When $\alpha =0$, it is straightforward to see that $X$ is monotone in $\lambda$ on trees. However, when $\alpha >0$, the question of monotonicity in $\lambda$ is less clear, even on trees, because speeding up red introduces more conversion opportunities.

Let $\mathbb N$ denote the positive integers $1,2,\hdots$ with root $x_0=1$, 
$S_n$ denote the star graph with root vertex $x_0$ attached to $n$ leaf vertices,
$\mathcal T$ denote a locally finite tree rooted at $x_0$, and $K_n$ denote the complete graph on $n$ vertices, labeled $1,2,\hdots, n$ with root $x_0 = 1$. See Figure~\ref{fig:graphs}. Note that $\mathbb N$ and $S_n$ are special cases of tree graphs.

We say that the number of damaged sites in chase-escape with conversion, $X$, is monotone in $\lambda$ if $X(G,x_0,\lambda',\alpha)\preceq X(G,x_0,\lambda,\alpha)$ for $\lambda' \leq \lambda$, and $X$ is monotone in $\alpha$ if $X(G,x_0,\lambda,\alpha')\preceq X(G,x_0,\lambda,\alpha)$ for $\alpha' \geq \alpha$. We say that $X$ is monotone in $n$ for the process on $S_n$ if $X(S_{n'},\lambda,\alpha,x_0)\preceq X(S_{n},\lambda,\alpha,x_0)$ for all $n' \leq n$, and similarly for monotonicity in $n$ for the process on $K_n$.

\input{graphs}
\begin{theorem}{\thlabel{thm:mono}} The number of damaged sites $X$ in chase-escape with conversion:

\begin{enumerate}[label = (\roman*)]

    \item Is monotone in $\alpha$ for $\mathcal T, \mathbb N$, $S_n$, and $K_n$.
 
    \item Is monotone in $\lambda$ for  $\mathbb N$, $S_n$, and $K_n$. 

    \item Is monotone in $n$ for $S_n$ and $K_n$.
    
\end{enumerate} 

\end{theorem}

We also prove that $\lambda_c(\alpha)$, defined at \eqref{eq:lc}, is non-trivial and pin down its leading order on infinite graphs with bounded degree and non-trivial site percolation threshold. Let $G$ be an infinite graph with root $x_0$. 
Recall that in Bernoulli site percolation, each vertex is open independently with probability $p$ and otherwise closed. The percolation critical value is $$p_c = p_c(G,x_0) := \inf \{ p \colon \P_p(x_0 \in \text{an infinite open cluster}) >0 \}.$$ 
Here a cluster is a maximal component of open vertices connected to one another by at least one edge.
We prove that there is a non-trivial phase transition whenever $p_c$ is non-trivial and, in doing so, obtain the first-order growth of $\lambda_c(\alpha)$. Note that for a non-negative function $f$ we write $f(x) = \Theta(x)$ if there exist constants $c,C>0$ such that $c < f(x)/x < C$ for all large $x$. 
\begin{theorem} \thlabel{thm:pt}
    Fix $\alpha >0$. Suppose that $G$ is an infinite graph with maximum degree $3\leq d < \infty$ and $p_c(G,x_0) <1$. It holds that
    $$\f\alpha{d-2}\leq \lambda_c(\alpha) \leq \frac{d+\alpha}{1 - p_c(G,x_0)^{1/d}} .$$ 
    Thus, %
    $\lambda_c(\alpha) = \Theta(\alpha)$.
\end{theorem}

Examples of graphs for which \thref{thm:pt} applies are infinite lattices and regular trees.

\subsection{Further questions}

We conjecture that $X$ is monotone in $\lambda$ on trees and $\mathbb Z^d$.  It would be worthwhile to study chase-escape with conversion on finite random networks whose topologies more closely resemble neuron structures, such as Erd\H{o}s-R\'enyi and random spatial graphs \cite{bernstein2021chase, thamattoor2012simulation}. A central question in chase-escape on $\mathbb Z^2$ is proving that $\lambda_c(0) <1$ \cite{si, cande}, or, even better, that $\lambda_c(0) <1/2$ as conjectured in \cite{kumar2021chase}. With the introduction of conversion, one could seek estimates on $\lambda_c(\alpha)$. For example, does $\lambda_c(\alpha)/\alpha \to C$ for some $C>0$ as $\alpha \to \infty$? Alternatively, it would be interesting to investigate the dual critical value $\alpha_c(\lambda) \coloneqq \inf \{ \alpha \colon \P_{\lambda, \alpha}(X< \infty) =1 \}$. 

In Figure~\ref{fig:crit}, we provide simulation estimates of $\lambda_c(1) \approx 1.975$ and $\alpha_c(1)\approx 0.275$ for the process on $\mathbb Z^2$. The curves in the figure also support our monotonicity on $\mathbb Z^d$ conjecture. See Figure~\ref{fig:clusters} for snapshots of the process on $\mathbb Z^2$. Another interesting variation would allow for an increasing conversion rate $\alpha(t)$ that models increased immune response over time.

A future project is to develop a data set using MRI data consisting of the spatial structure of individual MS lesions. Since MRI data is broken down into cubic millimeter boxes (known as voxels), storing a lesion's structure in $\mathbb Z^3$ would be natural. Such data could be used as benchmarks to compare models against. To our knowledge, none of the previously mentioned MS models have performed such benchmarking. The closest analogue is \cite{moise2021mathematical}, which compared their models to studies that measured the total volume of lesions in participants \cite{lublin2013randomized}, but this was non-spatial data.

\begin{figure}
    \centering
    \includegraphics[width=0.75\linewidth]{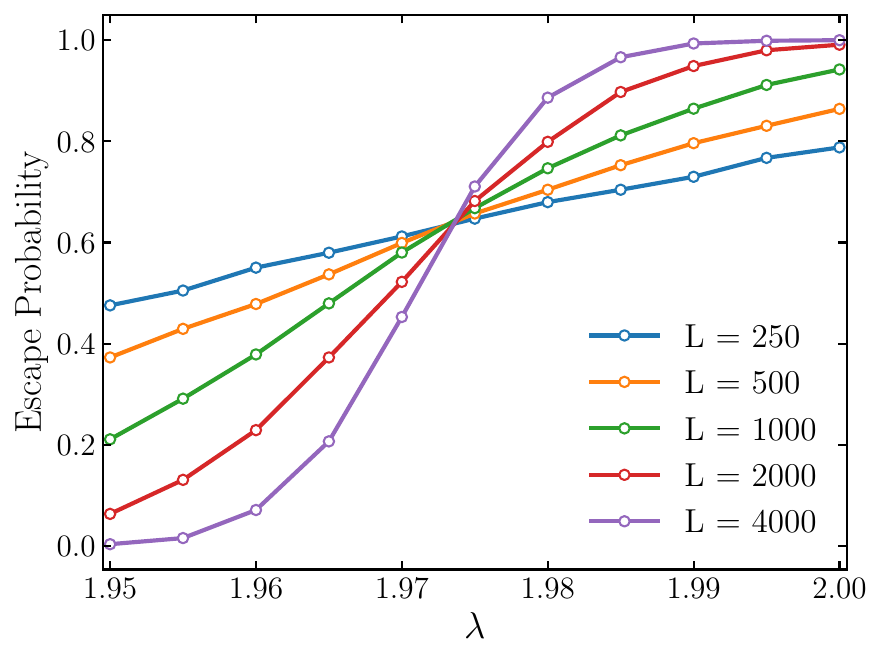}\hfill 
    \includegraphics[width=0.75\linewidth]{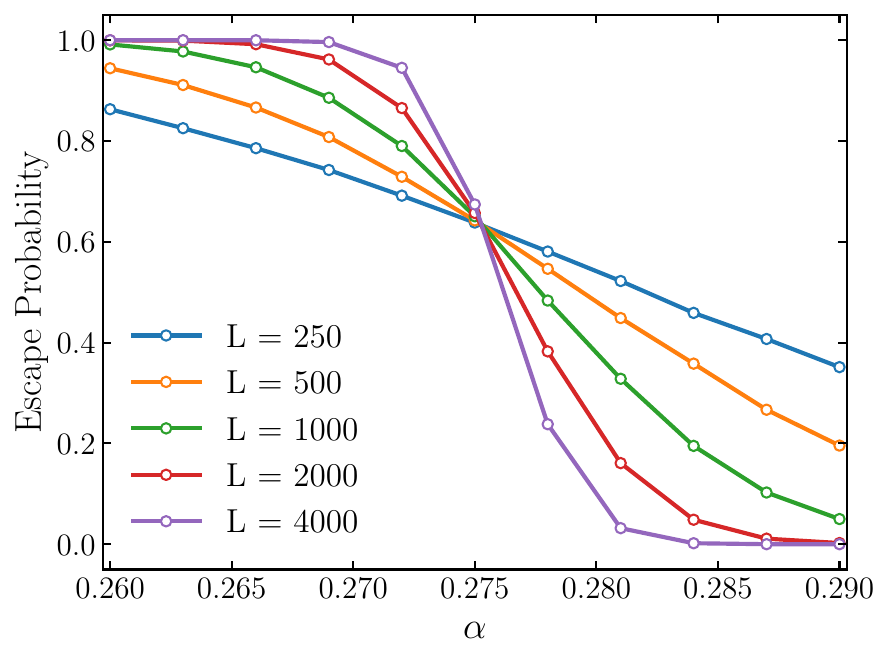}
    \caption{Simulations suggesting $\lambda_c(1) \approx 1.975$ ({top}) and $\alpha_c(1) \approx 0.275$ ({bottom}) for chase-escape with conversion on $\mathbb Z^2$. The process was run on two-dimensional tori with side lengths $L=250, 500, 1000, 2000, 4000$, initializing all vertices on the bottom edge as blue and those vertices one level above as red. The curves interpolate between empirical estimates of the ``Escape Probability", the probability that red reaches the top edge of the torus, averaged over 50,000 samples for each value of $\lambda$ and $\alpha$ indicated by open circles on the curves. The intersection of the curves for the different values of $L$ should be close to the true critical value, as in \cite{kumar2021chase}. Note that we also see what looks like monotonicity of the escape probability in $\lambda$ and $\alpha$.}
    \label{fig:crit}
\end{figure}

\begin{figure}
    \centering

    \includegraphics[width=0.75\linewidth]{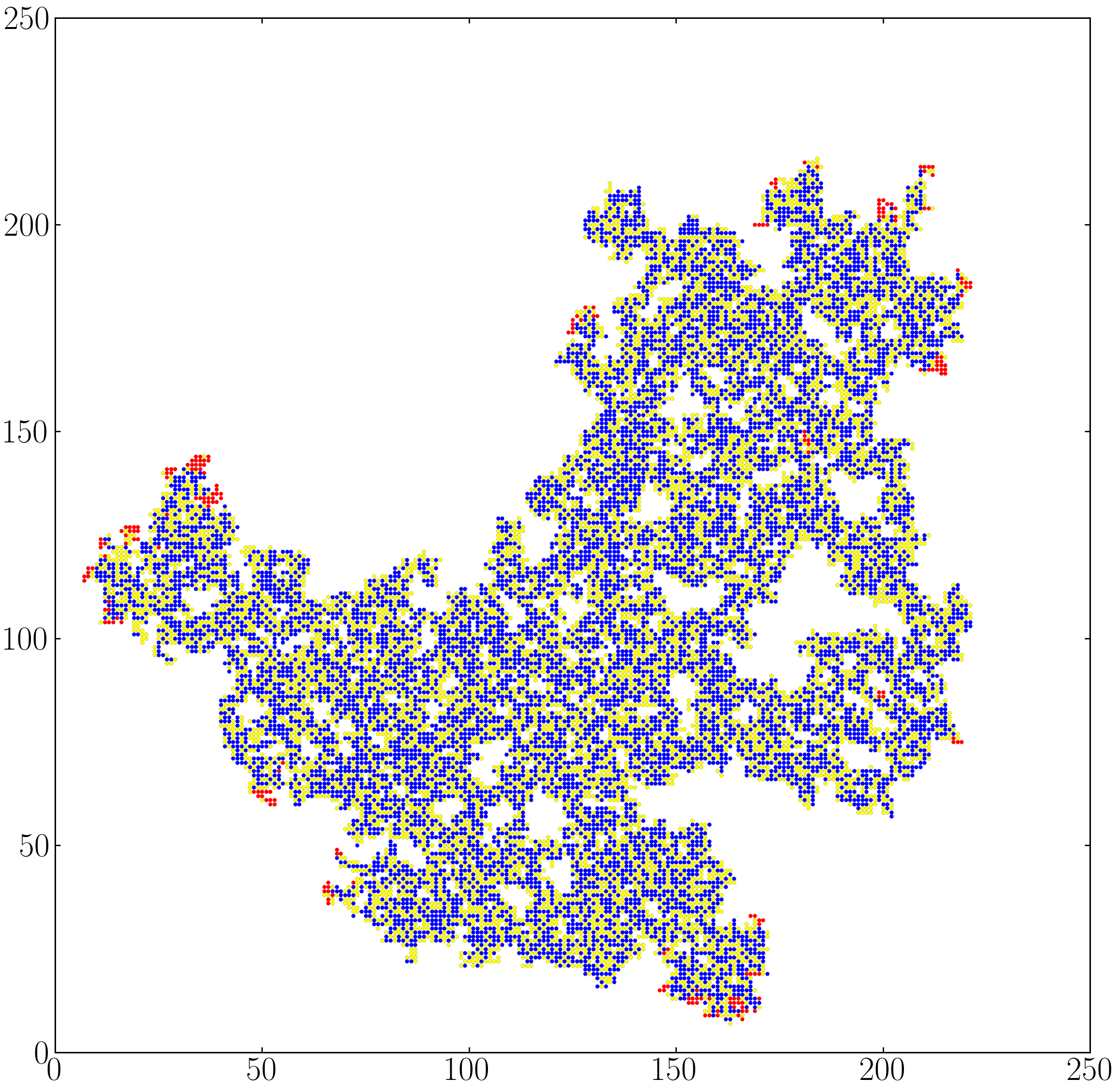}\hfill \includegraphics[width=0.75\linewidth]{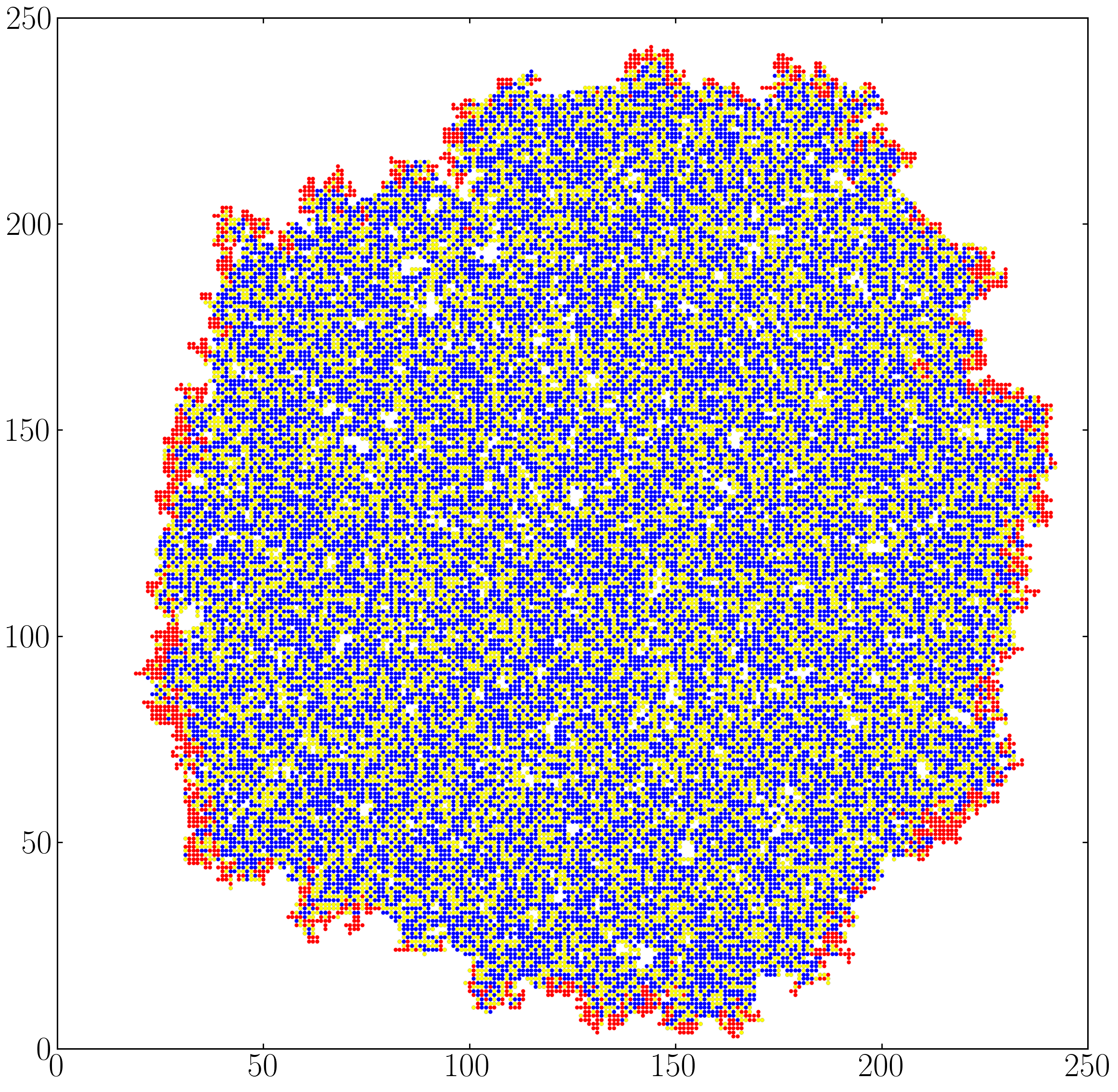}
    \caption{Sample realizations of the (still evolving) damaged region on a $250\times 250$ box 
    started with the central vertex in red and all other vertices in white.
    Sites that turned blue from conversion are colored yellow and those that turned blue from predation are colored blue. We set $\alpha=1$ in both with $\lambda = 1.976$ (top) and $\lambda =3$ (bottom). 
    }
    \label{fig:clusters}
\end{figure}

\subsection{Proof Overview}

\thref{thm:mono} is proven via explicit couplings that allow us to compare how $X$ is sampled for different parameter choices. Seemingly obvious monotonicity can be difficult to prove. The proper framing is needed to obtain \thref{thm:mono} and varies from graph to graph. For example, we find occasion to sample $X$ in a non-Markovian manner using edge passage times (for $\mathcal T$ and part of the proof for $\mathbb N$), in a continuous time Markovian manner (for $K_n$), in a mixture of these two viewpoints for $S_n$, and as a discrete time Markov process (for part of the proof for $\mathbb N$).

For $\mathbb N$, we discretize the process by generalizing a jump chain construction from \cite{beckman2021chase}. Conversion introduces new challenges. For example, the initial height of the jump chain is now random. To couple the starting locations of different jump chains, we use the passage time construction for trees. In the coupling, we must run the monotonically smaller process longer than its counterpart. Once the starting heights are coupled, comparing the evolution for different parameters is also more subtle due to conversion. Maintaining a working coupling requires the introduction of neutral flat steps to the jump chains. 

Monotonicity for the star graph uses a pure death process along with a queue to assign the times for blue to reach the root from a given leaf. 
For $K_n$ we use the framework introduced in \cite{complete} that reduces the dynamics to studying birth and death processes. Conversion introduces constant immigration to the previously pure birth process. The lower bound on $\lambda_c(\alpha)$ in \thref{thm:pt} is proven using a first-moment bound that relies on  red being unlikely to survive along a given path \cite{bernstein2021chase}. The upper bound uses a coupling with site-percolation that was introduced in \cite{gilbert}. 

\subsection{Organization}

We prove \thref{thm:mono} by graph: Section~\ref{sec:tree} handles trees,  Section~\ref{sec:Z} the positive integers, Section~\ref{sec:star} stars, and Section~\ref{sec:Kn} the complete graph. Section~\ref{sec:bounds} proves \thref{thm:pt}.

\section{Proof of \thref{thm:mono}: Trees} \label{sec:tree}

Fix a tree $\mathcal T$ with root $x_0$ as well as a value of $\lambda$. Let $\alpha' \geq \alpha$. We must prove that $X' \coloneqq X(\mathcal T, x_0, \alpha', \lambda) \preceq X(\mathcal T, x_0 ,\alpha, \lambda) =:X$. We will proceed by constructing the chase-escape with conversion process by assigning passage times to directed edges and vertices, then explaining how these times can be used to deduce the value of $X$ i.e., that $X$ is measurable with respect to these passage times. Specifically, we may partition $X$ by 
\begin{align*}
	X = \sum_{n=0}^\infty |\mathcal L(n)| = 1 + \sum_{n=0}^\infty \sum_{x\in \mathcal L(n)} |\mathcal I (x)|
\end{align*}
where $\mathcal L(n)$ denotes the set of vertices ever colored red at level $n$ and $\mathcal I(x)$ denotes the set of children of $x$ ever colored red. Note that given the tree structure, $\mathcal L(n+1) = \bigcup_{x\in \mathcal L(n)} \mathcal I(x)$. We will inductively define $\mathcal L(n)$ and $\mathcal I(x)$ using passage times. It then suffices to show that $\mathcal I'(x) \subset \mathcal I(x)$ for all $x \in \mathcal T$ where $\mathcal I'(x)$ is defined analogously to $\mathcal I(x)$ for the coupled process with parameter $\alpha' \ge \alpha$. 

Throughout the paper, we let $\Exp(\mu)$ denote the exponential distribution with density function $\mu x^{-\mu x}$ for $x \geq 0$. We will use the notation $Y \sim \Exp(\mu)$ to denote that $Y$ is a random variable with this distribution.  We proceed in two steps: first, describing the setup and then proving monotonicity.

\subsection{Setup}

Call $\hat y \in \mathcal T$ the \emph{parent} of $y$ if $\hat y$ shares an edge with $y$ and lies on the geodesic path of edges connecting $y$ to the root $x_0$. In this case, call $y$ a \emph{child} of $\hat y$. The \emph{progeny} of $y$ consists of all children of $y$, their children’s children, and so on.
Given $x \in \mathcal T$, define the subtree $\mathcal T(x)$ that includes vertices whose geodesic path to $x$ excludes $\hat x$, i.e., $\mathcal T(x)$ includes $x$ and all generations of its descendants. For each vertex $y \in \mathcal T(x)$, let $(x=x_1,x_2,\hdots, x_n= y)$ be the geodesic path of vertices between $x$ and $y$. This path is simply $(x)$ when $x=y$ and $(x,y)$ when $x$ and $y$ are adjacent.

To each edge $\mathcal E_{\hat y, y}$ we assign $(\vec R_{\hat y, y}, \vec B_{\hat y, y},\hat B_{\hat y, y}, C_y)$ consisting of four independent exponentially distributed random variables, 
which are also independent from those assigned to other edges. The values $\vec R_{\hat y, y} \sim \Exp(\lambda)$ and $\vec B_{\hat y, y} \sim \Exp(1)$ give the time it takes for red and blue, respectively, to spread from $\hat y$ to $y$. The value $\hat B_{\hat y, y} \sim \Exp(1)$ represents the time it takes for blue to spread from $y$ to $\hat y$. Note that the tree structure ensures that red cannot spread from $y$ to $\hat y$. Lastly,  $C_y \sim \Exp(\alpha)$ represents the conversion time at $y$. See Figure~\ref{fig:edges} for visual representation of these passage times as directed edges.

\begin{figure}
\begin{tikzpicture}[>=Stealth]
    \foreach \x in {0,1,2,3} {
        \node (\x) at (\x*2,0) {\x};
    }
    \node (4) at (4,-2) {4};
    \node (5) at (6,-2) {5};

    \draw[->,red] (0) -- node[midway, sloped, fill=white, inner sep=1pt] {\(\vec R_{0,1}\)} (1);
    \draw[->,blue,bend left=30] (0) to node[above] {\(\vec B_{0,1}\)} (1);
    \draw[->,blue,bend left] (1) to node[below] {\(\hat B_{1,0}\)} (0);

    \foreach \x [remember=\x as \lastx (initially 1)] in {2,3} {
        \draw[->,red] (\lastx) -- (\x);
        \draw[->,blue,bend left] (\lastx) to (\x);
        \draw[->,blue,bend left] (\x) to (\lastx);
    }

    \draw[->,red] (2) -- (4);
    \draw[->,red] (4) -- (5);
    \draw[->,blue,bend left] (4) to (5);
    \draw[->,blue,bend left] (5) to (4);
    \draw[->,blue,bend left] (2) to (4);
    \draw[->,blue,bend left] (4) to (2);

    \draw[->,purple] (0) to [out=130,in=50,loop] node[above] {\(C_0\)} ();
    \foreach \x in {1,2,3} {
        \draw[->,purple] (\x) to [out=130,in=50,loop] ();
    }
    \draw[->,purple] (4) to [out=310,in=230,loop] ();
    \draw[->,purple] (5) to [out=310,in=230,loop] ();
\end{tikzpicture}

\caption{Directed edges representing the passage times $\vec R_{\hat y, y}, \vec B_{\hat y, y}, \hat B_{y,\hat y}$, and $ C_y$. These are labeled for $\hat y = 0$ and are assigned similarly at the other vertices. On a tree, $X$ is measurable with respect to these passage times. 
}\label{fig:edges}
\end{figure}

Given a vertex $x$ and $y \in \mathcal{T} (x)$, define the infection times
\begin{equation}\label{eq:xyinfectiontime}
    T_{x,y} := \left(\sum_{i=1}^{n-1} \vec R_{x_i,x_{i+1}}\right) + C_y + \left(\sum_{i=1}^{n-1} \hat B_{x_{n-i+1},x_{n-i}}\right).
\end{equation}
The times $T_{x,y}$ give how long it takes for \emph{$x$ to be infected from the conversion of $y$}: red spreads from $x$ to $y$, the vertex $y$ converts to blue, and then the sites along the geodesic from $y$ to $x$ turn blue by predation until ultimately predating $x$ so it becomes blue.
For each vertex $x$, this sequence of events can occur at most once, for some $y$ in the progeny $\mathcal{T} (x)$ (note that $T_{x,y}$ considers only the conversion of vertex $y$). Accordingly, we define:
\[ 
    T_x = \min \{ T_{x,y} \colon y \in \mathcal T(x)\}.
\]

Let $\mathbf S_x$ denote the \emph{survival time of $x$}, that is, the time at which $x$ first turns blue. We use the conventions that bold letters are in the units of global time as the process runs and that $\mathbf S_x$ is infinite if $x$ never turns blue. Letting $(x_0, x_1,x_2,\hdots, x_n=x)$ be the geodesic path of vertices from $x_0$ to $x$, define 
\begin{align*}
	\mathbf R_x := \vec R_{x_0,x_1} + \cdots + \vec R_{x_{n-1}, x_n}
\end{align*}
to be the time for red to reach $x$ from $x_0$. 
We say that a vertex $x\in \mathcal T$ is of \emph{generation} $n$ if the geodesic path from $x_0$ to $x$ contains $n+1$ vertices.

For generation $0$, we have $\mathcal L(0) = \{x_0\}$. The root $x_0$ can only turn blue because one of its descendants (including itself) converted to blue and spread to $x_0$. Thus, we have
\begin{align*}
	\mathbf S_{x_0} := T_{x_0}.
\end{align*}
Letting $\mathcal{C}(x)$ denote the child vertices connected to the vertex $x$,  the children of the root $x_0$ that are ever in state red are exactly those who turned red before $x_0$ turns blue, i.e.,
\begin{align}
\mathcal I(x_0) \coloneqq \{x \in \mathcal C (x_0) \colon \vec R_{x_0, x} \leq \mathbf S_{x_0} \}. \label{eq:root}
\end{align}

Given $\mathcal L(0) , \dots, \mathcal L(n)$, $\mathbf S_x$ and $\mathcal I(x)$ for each $x\in \bigcup_{k\le n-1} \mathcal L(k)$, let us describe how to determine $\mathbf S_x$ and $\mathcal I(x)$ for each $x\in \mathcal L(n)$ and thus $\mathcal L(n+1)$. 

Fix $x \in \mathcal L(n)$ and consider its parent $\hat{x} \in \mathcal L(n-1)$. Note that $x$ turns blue either by predation from $\hat x$ or by infection from the conversion of one of its descendants (including itself). Thus, the survival time of $x$ is given by
\begin{align*}
	\mathbf S_x = (\mathbf R_x + T_x )\wedge (\mathbf S_{\hat x} + \vec B_{\hat x, x}).
\end{align*}

Thus, the set of children that $x$ converts to red is given by 
\begin{align}
\mathcal I(x) := \{ y \in \mathcal C(x) \colon \mathbf R_{y} \leq \mathbf  S_{x} \} \label{eq:I}
\end{align}
and
\begin{align*}
	\mathcal L(n+1) = \bigcup_{x\in \mathcal L(n)} \mathcal I(x).
\end{align*}

We obtain a partition $\{x_0\} \cup \bigcup_{n=0}^\infty \bigcup_{x \in \mathcal L(n)} \mathcal I(x)$ of vertices that are ever colored red. Using the memoryless property of the exponential distribution, this yields the following characterization of $X$
\begin{align*}
	X = 1+ \sum_{n=0}^\infty \sum_{x \in \mathcal L(n)} |\mathcal I(x)|.
\end{align*}

\subsection{Proof of monotonicity}  Fix $\alpha' \geq \alpha$, and define the analogous spreading times with the following couplings: $\vec R_{\hat y, y}'=\vec R_{\hat y,y}, \vec B_{\hat y, y}' = \vec B_{\hat y, y}, \hat B_{y,\hat y}' = \hat B_{y, \hat y}$, and  $C_y' \leq C_y$ with $C_y' \sim \Exp(\alpha')$. Using these coupled passage times and defining $\mathcal I'(x)$ and $\mathcal L'(n)$ analogously, we may sample $X'$ by taking 
\begin{align*}
	X' = 1 + \sum_{n=0}^\infty \sum_{x \in \mathcal L'(n)} |\mathcal I'(x)|.
\end{align*}
Using the convention that $\mathcal I(x) = \emptyset$ if $x\notin \bigcup_n \mathcal L(n)$, it suffices to show that $\mathcal I'(x) \subseteq \mathcal I(x)$ for all $x \in \mathcal T$. {We will do this by induction. For the base case, we show that $\mathcal{I}'(x_0) \subseteq \mathcal{I}(x_0)$. Recall that the red and blue passage times are coupled to be equal. Hence, the first and third terms in \eqref{eq:xyinfectiontime} are the same in both systems: 
\begin{align*}
\sum_{i=1}^{n-1} \vec R_{x_i,x_{i+1}} &= \sum_{i=1}^{n-1} \vec R'_{x_i,x_{i+1}}, \\
\sum_{i=1}^{n-1} \hat B_{x_{n-i+1},x_{n-i}} &= \sum_{i=1}^{n-1} \hat B'_{x_{n-i+1},x_{n-i}}.
\end{align*}
$C_y' \leq C_y$ implies $T'_{x,y} \leq T_{x,y}$ for all $y \in \mathcal{T}(x)$ and therefore 
$T_{x}' \leq T_{x}$ for all $x \in \mathcal{T}$. Thus, $\bm{S}_{x_0}' \leq \bm{S}_{x_0}$. Note that $\bm{R}_x = \bm{R}'_x$ for all $x \in \mathcal{T}$, so this ordering of the survival times imposes a more stringent condition at \eqref{eq:root} for inclusion in $\mathcal I'(x_0)$ than in $\mathcal I(x_0)$. Thus, $\mathcal I'(x_0) \subseteq \mathcal I(x_0)$.

Suppose we have $\mathcal I'(v) \subset \mathcal I(v)$ for every $v\in \mathcal T$ of generation $0$ to $n$. For any $x\in \mathcal T$ of generation $n+1$, its parent $\hat x$ is of generation $n$. The inclusion $\mathcal I'(x) \subset \mathcal I(x)$ trivially holds in the following scenarios:
\begin{itemize}
	\item $x \notin \mathcal I(\hat x)$: The induction hypothesis implies $x\notin \mathcal I'(\hat x)$ so that $\mathcal I(x) = \mathcal I'(x) = \emptyset$.
	\item $x \notin \mathcal I'(\hat x)$: Then $\mathcal I'(x) = \emptyset \subset \mathcal I(x)$.
\end{itemize}
Thus it remains to consider when $x \in \mathcal{I}(\hat x)\cap \mathcal{I}'(\hat x)$. In this case, both $\mathcal{I}(\hat x)$ and $\mathcal{I}'(\hat x)$ are nonempty and so $\bm S_{\hat x}, \bm{S}'_{\hat x}< \infty$ almost surely. Recall that
\begin{align*}
	\bm{S}_x &= (\bm{R}_x + T_x) \wedge (\bm{S}_{\hat x} + \vec B_{\hat x, x}),\\
	\bm{S}'_x &= (\bm{R}'_x + T'_x) \wedge (\bm{S}'_{\hat x} + \vec B'_{\hat x, x}).
\end{align*}
Since we have the equalities $\bm R_x = \bm R_x'$, $\vec B_{\hat x, x} = \vec B'_{\hat x, x}$, and the inequalities $T'_x \leq T_x$, $\bm S'_{\hat x} \leq \bm S_{\hat x}$, we can see that $\bm S'_{x} \leq \bm S_{x}$. Therefore, the condition in~\eqref{eq:I} is more stringent with $\bm S'_x$ than with $\bm S_x$ and so $\mathcal{I}'(x) \subseteq \mathcal{I}(x)$ as desired.

This proves that $\mathcal{I}'(x) \subseteq \mathcal{I}(x)$ for all $x \in \mathcal{T}$ and therefore that $X' \preceq X$ as desired. }
\begin{remark}
    We have proven the stronger statement that the set of vertices $\{x_0\} \cup \bigcup_{x \in \mathcal T} \mathcal I(x)$ ever colored red is stochastically decreasing in $\alpha$. 
\end{remark}
\begin{remark}
    This construction of $X$ highlights the difficulty in proving monotonicity in $\lambda$. Increasing $\lambda$, reduces the $\vec R_{x,y}$ passage times, and as a result, reduces both $\mathbf R_x$ and $\bm{S}_x$. Thus, both sides are reduced in the comparison in \eqref{eq:I}. 
\end{remark}

\begin{remark}
    This argument does not generalize to graphs with cycles because the comparisons needed to determine which vertices are ever colored red are more complicated.
\end{remark}

\section{Proof of \thref{thm:mono}: The positive integers} \label{sec:Z}

As observed in \cite{beckman2021chase}, a useful quantity for analyzing variants of the chase-escape process on the positive integers $\mathbb N$ is the \emph{jump chain} $(Y_t)_{t \in \mathbb N}$.  This discrete-time process tracks the number of (necessarily contiguous) red sites in front of the rightmost blue. For example, the jump chain corresponding to this configuration of red and blue on $\mathbb N$
\quad 
\\ 

\begin{center}

  \begin{tikzpicture}[scale = .5]
            \node[circle,fill=red, draw = black, inner sep = .1 cm] at (0, 0) {};
            \foreach \n in {1,...,8}{
                \node[circle,fill=red, draw = black, inner sep = .1 cm] at (\n, 0) {};
                \draw (\n -.75,0) -- (\n-.25,0);
            }    
            \foreach \n in {9,...,10}{
                \node[circle,fill=white, draw = black, inner sep = .1 cm] at (\n, 0) {};
                \draw (\n -.75,0) -- (\n-.25,0);
            }    
            \node[circle,fill=blue, draw = black] at (4,0) {};
            \node[circle,fill=blue, draw = black] at (5,0) {};
            \node[circle,fill=blue, draw = black] at (1,0) {};
            \node[circle,fill=blue, draw = black] at (2,0) {};
            \draw[|-|] (6,-.75) -- (8,-.75);
            \node at (7,-1.25) {$Y_t$};
            
        \end{tikzpicture}

\end{center}
would have height 3. If, at the next transition, the configuration changed to\quad 
\\ 

\begin{center}

  \begin{tikzpicture}[scale = .5]
            \node[circle,fill=red, draw = black, inner sep = .1 cm] at (0, 0) {};
            \foreach \n in {1,...,8}{
                \node[circle,fill=red, draw = black, inner sep = .1 cm] at (\n, 0) {};
                \draw (\n -.75,0) -- (\n-.25,0);
            }    
            \foreach \n in {9,...,10}{
                \node[circle,fill=white, draw = black, inner sep = .1 cm] at (\n, 0) {};
                \draw (\n -.75,0) -- (\n-.25,0);
            }    
            \node[circle,fill=blue, draw = black] at (4,0) {};
            \node[circle,fill=blue, draw = black] at (5,0) {};
            \node[circle,fill=blue, draw = black] at (7,0) {};
            \node[circle,fill=blue, draw = black] at (1,0) {};
            \node[circle,fill=blue, draw = black] at (2,0) {};
            \draw[|-|] (8,-.75) -- (8,-.75);
            \node at (8,-1.25) {$Y_{t+1}$};
            
        \end{tikzpicture}

\end{center}
the jump chain would then have height $1$. 
We emphasize that the values $t=0,1,\hdots$ are discrete and correspond to the (random) times at which the configuration changes. See \cite{beckman2021chase} for a more formal description in the case of chase-escape.

For our purposes, an \emph{active jump chain} is then described by a lattice path that:
\begin{itemize}
    \item Starts at $(0,Y_0)$ for some $Y_0 \geq 0$. 
    \item Consists of $+(1,1)$ up-steps and $+(1,-j)$ down-steps for any $j\geq 1$.

    \item Is never negative and stops upon reaching 0. 
\end{itemize}

If the jump chain is at $(t,k)$, then the next step is to $(t+1,j)$, for $j \in \{0,1,\hdots,k-1, k+1\}$, with probabilities $p_{k,j}$ where
\[
    p_{k,k+1} = \f{ \lambda }{1 + \lambda + \alpha k}, \qquad  p_{k,k-1}= \f{ 1 + \alpha}{1+\lambda+\alpha k},
\]
and 
\[ 
    p_{k,j} = \f{\alpha}{1+\lambda + \alpha k} \qquad 0 \leq j \leq k-2. 
\]
The first term, $p_{k,k+1}$, is the probability that the rightmost red advances. The second term, $p_{k,k-1}$, is the probability that either the rightmost blue advances or the leftmost red (in the relevant connected component) converts to blue. 
The third term, $p_{k,j}$, is the probability that the red at distance $k-j$ from the rightmost red converts to blue.
Reaching 0 corresponds to the rightmost red becoming blue, thus ending the possibility of red advancement.

The jump chain evolves until the stopping time $\kappa = \inf\{t \colon Y_t = 0 \}$ with the convention that $\kappa = 0$ whenever $Y_0 =0$. Let
\begin{align*}
	U(n) = \# \{0 < t \leq \kappa \colon Y_{t} = Y_{t-1}+1 \mid Y_0 = n\}
\end{align*}
be the number of up-steps by the jump chain $(Y_t)_{t \in \mathbb{N}}$ starting with $Y_0 = n$.  

One can start tracking the jump chain at any point from the appearance of the first blue on. Let $N_t$ be the site of the rightmost red at time $t$ and $M_t$ be the site of the rightmost blue at time $t$ with the convention that $N_t = M_t$ if the rightmost non-white site is blue, the jump chain construction yields the following characterization of $X$ on the positive integers $\mathbb{N}$, rooted at $x_0 = 1$.

\begin{lemma} \thlabel{lem:X} For any time $t$ when blue is present in the system, $X(\mathbb N, x_0, \alpha, \lambda) \overset{d}{=} N_t + U(N_t-M_t)$.
\end{lemma}

In particular, let $\gamma$ be the time of the first conversion from red to blue and denote $N := N_\gamma$ and $M := M_\gamma$, that is, $M$ is the first site to convert to blue. \thref{lem:X} implies that $X(\mathbb N, x_0, \alpha, \lambda) \overset{d}{=} N + U(N-M)$.

We will prove that $X$ is monotone in $\alpha$ and $\lambda$ simultaneously by coupling the jump chains corresponding to parameter choices $\alpha \leq \alpha'$ and $\lambda \geq \lambda'$. Let $X' = X'(\mathbb N, x_0, \alpha', \lambda')$ be sampled from the jump chain $(Y_t')_{t \in \mathbb N}$ with parameters $\lambda'$ and $\alpha'$. Define $N'_t$, $M'_t$, and $U'$ analogously for the $(\lambda', \alpha')$-chase-escape with conversion process. 
We will prove the following characterization of $X'$ that, in light of \thref{lem:X}, implies that $X' \preceq X$.

\begin{lemma}\thlabel{lem:X'}
    There exists $\gamma'$ such that $X'(\mathbb N, x_0, \alpha', \lambda')  \overset{d}{=} N'_{\gamma'} + U'(N'_{\gamma'}-M'_{\gamma'})$ with the relations:
    \begin{enumerate}[label = (\roman*)]
        \item $N'_{\gamma'} \preceq N$ 
        \item $N'_{\gamma'}- M'_{\gamma'} \preceq N- M$
        \item $U'(n') \preceq U(n)$ for all $n' \leq n$.
    \end{enumerate}
\end{lemma}

\begin{proof}
    We will proceed in two steps. First, we will produce a coupling that uses the passage time construction for trees from the previous section to prove claims (i) and (ii). The basic idea is that we take $\gamma'$ to be the time the site $M$ becomes blue in the $(\lambda', \alpha')$-process. %

    In Step 2, we will proceed inductively to prove that, conditional on $Y_t' \leq Y_t$, the next step of the chain results in $Y_{t+1}' \leq Y_{t+1}$. This implies that $U'(n') \preceq U(n)$. The second step requires a {minor} time distortion where sometimes only one of the jump chains takes an up- or down-step while the other takes a \emph{flat-step}. However, this does not alter the total number of up-steps taken.

    \textbf{Step 1: (i) and (ii).} First we will formally sample $M$ and $N$ used to determine $X$ in \thref{lem:X}. As $\mathbb N$ is a tree graph, the passage time construction from Section~\ref{sec:tree} applies. In particular, for $y \geq 1$ we sample times $\vec R_{y, y+1}, \vec B_{y ,y+1},  \hat{B}_{y+1, y}$, and $C_y$. 
    Let $\mathbf R_1 = 0$ and, for $x \geq 2$, define $\mathbf R_x = \mathbf R_{x-1} + \vec R_{x-1,x}$ to be the time it takes for red to reach site $x$ in the absence of chase and conversion.  Note that $\mathbf R_x+ C_x$ is the time it takes for site $x$ to become red, and then convert. Define $\gamma = \min \{ \mathbf R_x + C_x \colon x \in \mathbb N \}$ as the first conversion time. This minimum is realized at some almost surely unique site $M$. We set $N = \max \{ x \colon T_x \leq  \gamma \}$ to be the location of the rightmost red at time $\gamma$ with the convention that $N=M$ if $M$ is the rightmost non-white site. 

    For $\alpha' \geq \alpha$ and $\lambda' \leq \lambda$, sample the edges $\vec R_{y, y+1}', \vec B_{y ,y+1}',  B_{y+1, y}'$, and $C_y'$ so that the blue spreading times are the same and $\vec R_{y, y+1}'\geq \vec R_{y, y+1}$ and $C_y' \leq C_y$. Define $\mathbf R_x'$ analogously using these passage times. Taking $M$ exactly as in the previous paragraph (using the $\lambda$ and $\alpha$ passage times), let $\gamma' = \mathbf R_M' + C_M'$ be the time at which site $M$ would convert to blue using the $\lambda'$ and $\alpha'$ passage times. 
    We let $M'_{\gamma'}$ be the site with the rightmost blue at time $\gamma'$ in the $(\lambda', \alpha')$-process. 
    {Note that $M'_{\gamma'}$ need not equal $M$ nor the first site to turn blue in the $(\lambda', \alpha')$-process}; it is possible that other sites were converted earlier, and even that some sites became blue from predation. We let $N'_{\gamma'}$ be the location of the rightmost red at time $\gamma'$ in the $(\lambda', \alpha')$-process with the convention that $N'_{\gamma'}=M'_{\gamma'}$ if $M'_{\gamma'}$ is the rightmost non-white site. See~Figure \ref{fig:Y0} for a visual representation of these quantities.

To derive the relationship between these variables, we will first deal with the case where site $M$ is white at time $\gamma'$ in the $(\lambda', \alpha')$-process, followed by the case when  $M$ is blue at time $\gamma'$ in the $(\lambda', \alpha')$-process. The two cases are represented in the right and left panels, respectively, of Figure~\ref{fig:Y0}. These are the only cases, since $M$ cannot be red at time $\gamma'$, given that $\gamma'$ is the time that site $M$ would convert from red to blue. 

If $M$ is white at time $\gamma'$, then site $M$ remains white for all time in the process with the $(\lambda', \alpha')$ passage times, and a blue site strictly to the left of $M$ is the rightmost non-white site. 
Therefore, $M'_{\gamma'} = N'_{\gamma'} < M \leq N$. Therefore, $N'_{\gamma'} \leq N$ and $0 = N'_{\gamma'}-M'_{\gamma'} \leq N-M$ in this case. 

Assume now that site $M$ is blue at time $\gamma'$. In this case, $M'_{\gamma'} \geq M$, because $M'_{\gamma'}$ is the rightmost blue site and site $M$ is blue. Since $N$ is the location of the rightmost red site in the $(\lambda, \alpha)$-process, it follows that
\[
    C_M < \sum_{i=0}^{N-M} \vec R_{M+i,M+i+1},
\]
that is, the time that it takes red to spread from site $M$ to site $N+1$ (at rate $\lambda$) is greater than the time it takes site $M$ to convert to blue, at rate $\alpha$. This, combined with the passage time relations $\vec R_{y, y+1}'\geq \vec R_{y, y+1}$ and $C_y' \leq C_y$, implies
\begin{align*}
    \gamma' & = \bm R'_M + C'_M\\
            &\leq \bm R'_M + C_M\\
            &< \bm R'_M + \sum_{i=0}^{N-M} \vec R_{M+i,M+i+1}\\
            &\leq \bm R'_M + \sum_{i=0}^{N-M} \vec R'_{M+i,M+i+1}\\
            &= \bm R'_{N+1} .
\end{align*}
Therefore, $\gamma' < \bm R'_{N+1}$ and so site $N+1$ is white in the $(\lambda', \alpha')$-process at time $\gamma'$. This implies that $N'_{\gamma'} \leq N$. Since $M'_{\gamma'} \geq M$, we have moreover $N'_{\gamma'}- M'_{\gamma'} \leq N-M$. This gives (i) and (ii).
\input{Y0.tex}

    \textbf{Step 2: (iii).} Let $t \geq 0$ and suppose that $Y_t = y \geq y' = Y_t'$ with $\lambda \geq \lambda'$ and $\alpha \leq \alpha'$. We will sample the next step of these chains using the Poisson edge clocks associated to the current chase-escape with conversion configuration on $\mathbb N$. 
    
    Adopting the convention that an $\Exp(0)$-distributed random variable is infinite with probability one, sample the following random variables independently:
\begin{itemize}
    \item $\tau_B \sim \Exp(1)$.
    \item $\tau_R \sim \Exp(\lambda')$ and $\sigma_R \sim \Exp(\lambda - \lambda')$. 
    \item $\tau_i \sim \Exp(\alpha)$ for $i=1,\hdots, y$, and $\sigma_i \sim \Exp(\alpha' - \alpha)$ for $i=1,\hdots, y'$. 
\end{itemize}

We then set the edge transition times:
\begin{itemize}
    \item $T_B = \tau_B = T_B'$.
    \item $T_R = \min( \tau_R, \sigma_R)$ and $T_R' = \tau_R$.
    \item $T_{R,i} = \tau_i$ for $i=1,\hdots, y$ and $T_{R,i}' = \min(\tau_i, \sigma_i)$ for $i = 1,\hdots, y'$.
\end{itemize}

Let $$\tau = \min\{ \tau_B, \tau_R, \sigma_R, \tau_1, \hdots, \tau_y, \sigma_1, \hdots, \sigma_{y'}\}.$$ Call the times $T_B, \hdots, T_{R,y}$ the \emph{$Y_t$ edge clocks}. Similarly, the times $T_B', \hdots, T_{R,y'}'$ are the \emph{$Y_t'$ edge clocks}. 

To obtain $Y_{t+1}$ and $Y_{t+1}'$, we update the edge(s) corresponding to the edge clock(s) equal to $\tau$. If only one edge clock updates, then the other process takes a flat step. We give the details of each situation:

\begin{itemize}
    \item If $\tau = \tau_B$ (hence $\tau = T_B = T_B'$), then for both $Y_t$ and $Y_t'$, we set the blue vertex of their respective critical regions to predate the adjacent red and so $Y_{t+1}=y-1 \geq y'-1 = Y_{t+1}'$. 
    \item If $\tau = \tau_R$ (hence $\tau = T_R = T_R'$), then both $Y_t$ and $Y_t'$ take a \emph{forward step} and infect their adjacent white vertices. Thus, $Y_{t+1} = y_t + 1 \geq y_t' + 1 = Y_{t+1}'$. 
    \item If $\tau = \sigma_R$, (hence $\tau = T_R$) then we set $Y_t$ to take a \emph{forward step} and $Y_t'$ to take a \emph{flat step}. We have that $Y_{t+1} = y + 1 > y' = Y_{t+1}'$. 
    \item If $\tau = \tau_i$ for $1 \leq i \leq y'$ (hence $\tau = T_{R,i} = T_{R,i}')$, then for both jump chains the $i^\textit{th}$-red vertex (counting from the rightmost red vertex in each chain as vertex $1$) converts to blue. This means that $Y_{t+1}=i-1=Y_{t+1}'$. 
    \item If $\tau = \tau_i$ for $y' < i \leq y$ (hence $\tau = T_{R,i}$) we set that the $i^\textit{th}$-red vertex of $Y_t$ converts to blue (again counting from the right) and $Y_t'$ takes a \emph{flat step} (notice that we only defined $T_{R,i}'$ for $i = 1, \dots, y'$, so in this case, no conversion happens in the $Y'_t$ chain). Hence, $Y_{t+1} =i-1 \geq y'=Y_{t+1}'$.
    \item If $\tau = \sigma_i$ for $ 1 \leq i \leq y'$ (hence, $\tau = T_{R,i}'$), then we set the $i^\textit{th}$-red vertex of $Y_t'$ to convert while $Y_t$ takes a \emph{flat step}. Hence $Y_{t+1} = y > i-1 = Y_{t+1}'$. 
\end{itemize}
See~Figure~\ref{fig:Y} for an example.

\begin{figure}
    \centering
 \begin{tikzpicture}[scale = .8]
    \draw[step=1cm,gray,very thin] (0,0) grid (12,5);
    \draw[dashed, thick] (0,1)
        -- ++(1,1)
        -- ++(1,1)
        -- ++(1,0)
        -- ++(1,1)
        -- ++(1,1)
        -- ++(1,-3)
        -- ++(1,0)
        -- ++(1,1)
        -- ++(1,-1)
        -- ++(1,0)
        -- ++(1,-2);
    \draw[thick] (0,2)
        -- ++(1,1)
        -- ++(1,1)
        -- ++(1,-1)
        -- ++(1,1)
        -- ++(1,1)
        -- ++(1,0)
        -- ++(1,-2)
        -- ++(1,1)
        -- ++(1,-2)
        -- ++(1,1)
        -- ++(1,0)
        -- ++(1,-3);
\end{tikzpicture}
    \caption{An example of the coupling in Step 2. The solid black path corresponds to $Y_t$ with $Y_0=2$ and the dashed path to $Y_t'$ with $Y_0'=1$. Our construction ensures $Y_t' \leq Y_t$ at all steps so long as $Y_0' \leq Y_0$.}
    \label{fig:Y}
\end{figure}
\begin{claim}
    Let $Y$ be the total number of upward steps taken by $(Y_t)_{t \in \mathbb{N}}$ until reaching 0, and similarly for $Y'$. It holds that $Y \succeq Y'$. 
\end{claim}

\begin{proof}

   Once the flat steps are removed, this coupling gives the correct marginals for $Y_t$ and $Y_t'$ conditional on the values of $Y_0$ and $Y_0'$. Moreover, whenever $Y'_t$ increases (when $\tau = \tau_R$), $Y_t$ is coupled to increase by the same amount. Whenever $Y_t$ decreases, the coupling above guarantees that the $Y_t$ chain does not fall below $Y'_t$. %
   We also have cases when $Y_t'$ moves downwards and $Y_t$ stays fixed, and $Y_t$ moves upwards with $Y'_t$ fixed. Because the coupling ensures that the chains remain ordered and couples the forward steps, it is immediate that $U'_\kappa(n') \preceq U_{\kappa}(n)$ whenever $n' \leq n$. 
\end{proof}

Recall that~\thref{lem:X} states that $X(\mathbb N, x_0, \alpha, \lambda) = N + U_\kappa(N-M)$ and \thref{lem:X'} that $X' = N'_{\gamma'} + U_{\kappa'}'(N'_{\gamma'}-M'_{\gamma'})$ with stochastic relations: $N'_{\gamma'} \preceq N$, $N'_{\gamma'} - M'_{\gamma'} \preceq N-M$, and $U_{\kappa'}' \preceq U_{\kappa}(n)$ for all $n' \leq n$. It follows that
\[
    X'(\mathbb N, x_0, \alpha', \lambda') = N'_{\gamma'} + U'_{\kappa'}(M'_{\gamma'}-N'_{\gamma'}) \preceq N + U_\kappa(N-M) 
                                          = X(\mathbb N, x_0, \alpha, \lambda),
\]
as desired.
\end{proof}

\section{Proof of \thref{thm:mono}: Stars} \label{sec:star}

For the star graph, we will return to thinking about the process in continuous time, $t\in [0,\infty)$, with the exponential clocks framework. For the star graph, the root is initially red, and red can only spread from the root. Accordingly,  the number of leaves that are ever red is equal to the number of leaves in state red or blue at the time the root becomes blue. 

In the following, we think of the spread of red from the root like a pure death process that starts with $n$ individuals (representing the $n$ white nodes), and where each one dies (here, this corresponds to becoming red) independently at rate $\lambda$. Let $\sigma(i) \in [0,\infty)$ denote the time of the $i$-th death with the convention that $\sigma(0) = 0$ and $\sigma(n+1) = \infty$. Note that $\sigma(i) - \sigma(i-1) \sim \Exp( (n-i +1) \lambda)$ for $1 \leq i \leq n$. 

Similarly, let $\sigma'(i)$ denote the time of the $i$-th death in a coupled pure death process that starts with $n' \leq n$ individuals that each die at rate $\lambda ' \leq \lambda$. Therefore, $\sigma'(i) - \sigma'(i-1) \sim \Exp((n'-i)\lambda')$. Because $n' \leq n$ and $\lambda' \leq \lambda$, for each $0 \leq i \leq n'$, it is clear that $(n'-i)\lambda' \leq (n-i)\lambda$. Therefore, we couple these two processes such that $$\sigma'(i) - \sigma'(i-1) \geq \sigma(i)-\sigma(i-1)$$ for each $0 \leq i \leq n'$. Notice that this also implies that $\sigma'(i+1) - \sigma'(k) \geq \sigma(i+1) - \sigma(k)$ for any $k \leq i \leq n'$.  For $\alpha' \geq \alpha$, let $T_i \sim \Exp(\alpha) + \Exp(1)$ be independent and coupled with $T_i' \sim \Exp(\alpha') + \Exp(1)$ so that $T_i' \leq T_i$. $T_i$ is the time it takes for the $i$th red leaf to convert to blue and then spread blue to the root (similarly for $T_i'$). Additionally, let $T_0\sim \Exp(\alpha)$ and $T_0' \sim \Exp(\alpha')$ be coupled so that $T_0' \leq T_0$. $T_0$ and $T_0'$ represent the time it takes the root to convert to blue. 

For $0 \leq i \leq n$, let 
\[
    M_i = \min_{0 \leq k \leq i} \sigma(k) + T_k
\] 
and analogously for $M_i'$. 
Define
\begin{align}
I & = 
\min\{ 0 \leq i \leq n \colon \sigma(i+1) - M_i > 0\},\label{eq:Xm}  \\
I ' & = 
\min\{ 0 \leq i \leq n' \colon \sigma'(i+1) - M'_i > 0\}.  
\end{align}
We claim that $X  = I + 1$ and $X'  = I' + 1.$
To see this, note that the times $\sigma(i)$ and $T_i$ can be coupled to reflect the time red spreads to the $i$th leaf and the time the $i$th red leaf would convert the root to blue, as described above. The smallest index $I$ such that that $\sigma(I+1) > M_I$, equivalently $\sigma(I+1) - M_I >0$, indicates that the root has been colored blue (by one of the first $I$ red leaves or by conversion) before the $(I+1)$th leaf turns red. So $I$ gives how many leaves are at some point colored red and $X=I+1$ to account for the root. Similar reasoning gives the characterization of $X'=I'+1$. 

We may rewrite $\sigma(i+1) - M_i$ as follows
\begin{align}
    \sigma(i+1) - M_i &= \sigma(i+1) - \min_{0 \leq k \leq i} (\sigma(k) - T_k)\\
    &= \max_{0 \leq k \leq i} (\sigma(i+1) - \sigma(k) - T_k) \label{eq:alt}
\end{align}
Recall that from our coupling, we have that for all $0 \leq k \leq i$
$$\sigma(i+1) - \sigma(k) \leq \sigma'(i+1) - \sigma'(k), $$
so 
$$\sigma(i+1) - \sigma(k) - T_k \leq \sigma'(i+1) - \sigma'(k) - T_k.$$
Similarly by the coupling, because $T_k \geq T_k'$, we also have that 
$$\sigma(i+1) - \sigma(k) - T_k \leq \sigma'(i+1) - \sigma'(k) - T_k'.$$
Therefore, for $0 \leq i \leq n'$,
\begin{align*}
    \max_{0 \leq k \leq i} (\sigma(i+1) - \sigma(k) - T_k )&\leq \max_{0 \leq k \leq i} (\sigma'(i+1) - \sigma'(k) - T_k'),
\end{align*}
which by \eqref{eq:alt} implies that
\begin{align}
    \sigma(i+1) - M_i & \leq \sigma'(i+1)-M'_i. \label{eq:M<M'}
\end{align}

Suppose now that $I = i$ with $0 \leq i \leq n'$. Then $0 < \sigma(i+1) - M_i$, and by \eqref{eq:M<M'} 
$$0 < \sigma(i+1) - M_i \leq \sigma'(i+1) - M'_i.$$
which yields $I' \leq i = I$. On the other hand, if $I = i$ with $i > n'$, then $I' \leq n' < i = I$. Since $X = I+1$ and $X'=I'+1$, this proves that $X' \preceq X$. 

\begin{remark}
    We point out that the passage time for construction of $X$ on trees in Section~\ref{sec:tree} does not appear to easily give monotonicity of $X$ in $\lambda$ and $n$ if one simply couples the passage times for the $(\lambda, \alpha, n)$ and $(\lambda', \alpha',n ')$ processes in the canonical way. Without the further coupling we give here, of assigning the conversion time and then predation of the root dynamically, it is possible that speeding up an edge passage time or adding a new edge could dramatically reduce $X$ compared to $X'$. 
\end{remark}

\section{Proof of \thref{thm:mono}: The complete graph} \label{sec:Kn}

 As with $S_n$, it is more convenient to work in continuous time $t\in [0,\infty)$ using the exponential clocks in the definition of chase-escape with conversion.  Let $(R_t)_{t \in \mathbb N}$ be the number of vertices in state $r$ at time $t$. Similarly, let $(B_t)_{t \in \mathbb N}$ be the number of blue vertices. Let $W_t = n - R_t - B_t$ be the number of vertices in state $w$. Conditional on $R_t = r, B_t=b,$ and $W_t = w = n - (r +b)$, we have the following transition probabilities at the next jump time:

$$p^+_{r,b} = \f{ \lambda w r}{\lambda r w + b r + \alpha r}; \quad p^-_{r,b} = \f{br + \alpha r}{\lambda r w + b r + \alpha r},$$
where $p^+_{r,b}$ is the probability $R_t$ increases by 1 and $p^-_{r,b}$ is the probability $R_t$ decreases by one while $B_t$ increases by 1. As observed for chase-escape on the complete graph in \cite{complete}, $r$ is common to all summands and can thus be factored and canceled. The common $r$ term persists with conversion. Hence, for all $1 \leq r \leq n$ we have

\begin{align} p^+_{r,b} := \f{ \lambda w }{\lambda  w + b  + \alpha }; \quad p^-_{r,b}:= \f{b + \alpha }{\lambda w + b  + \alpha }. \label{eq:p+-}\end{align}

\cite[Theorem 1.1]{complete} uses this observation to deduce that the number of sites ever infected by red can be characterized by a simple process involving independent pure birth and pure death processes. This is still true in chase-escape with conversion, but with the addition of constant immigration to the birth process. We describe these processes below.

The death process starts with $n-1$ individuals, each of which die at exponential rate $\lambda$. Let $\mathcal W_t$ denote the number of individuals remaining at time $t$. The birth process starts with 0 individuals who generate an additional individual at exponential rate $1$. Additionally, there is a constant immigration process where a new individual is added at exponential rate $\alpha$ (that does not depend on the population size). Let $\mathcal B_t$ denote the number of individuals at time $t$. Additionally, let $\sigma(i)$ denote the time (in continuous units) of the $i$th jump in the death process and $\rho(i)$ denote the time of the $i$th jump in the birth process. By construction,  $\sigma(i+1) - \sigma(i) \sim \Exp( \lambda (n-1-i))$ for $0 \leq i \leq n-2$ and $\rho(i+1) - \rho(i) \sim \Exp(i + \alpha)$ for $i \geq 0$.

One can check that $\mathcal W_t$ at its jumps corresponds to $W_t$ and $\mathcal B_t$ tracks $B_t$ (with a time change, since we have removed the scaling corresponding to $R_t$). This is true because the transition probabilities for either chain are given by \eqref{eq:p+-}. Letting $$\rho_*:= \rho( \min \{ i \geq 1 \colon \rho(i) < \sigma(i) \}),$$ this correspondence holds up to the stopping time
$$\tau = \sigma(n-1) \wedge \rho_*$$
at which point either no white or no red vertices remain. We can use this coupling to characterize $X$ as:
    $$X(K_{n}, x_0, \alpha, \lambda) =  n - \mathcal W_\tau.$$
We obtain monotonicity of $X$ in $\lambda$, $\alpha$, and $n$ by virtue of this coupling. To go into more detail, increasing $\lambda$ speeds up the jump times $\sigma(i)$, this stochastically increases $\rho_*$, thus decreasing $\mathcal W_\tau$ and increasing $X$. On the other hand, increasing $\alpha$ speeds up the jump times $\rho(i)$, which stochastically decreases $\rho_*$, thus increasing $\mathcal W_\tau$ and decreasing $X$. Lastly, increasing $n$ speeds up the jump times $\sigma(i)$, thus this increases $\rho_*$, which results in a stochastic increase to $X$. 

\begin{remark}
    An interesting future inquiry is generalizing the results from \cite{complete} to chase-escape with conversion. The main difference is the constant immigration at rate $\alpha$, which appears to have a non-trivial impact on the key tool that pure birth and death processes are representable as time changes of unit Poisson processes.
\end{remark}

\section{Proof of \thref{thm:pt}} \label{sec:bounds}

In this section, we will take $\alpha>0$ to be fixed and indicate the dependence of the measure $\P$ on $\lambda$ with a subscript $\P_\lambda$. First we prove the lower bound. We proceed by showing the stronger statement that $\E_\lambda[X]<\infty$ for sufficiently small $\lambda$. %
Let $\Gamma_k$ be the set of all vertex self-avoiding paths of length $k$ starting at $x_0$ that are present in $G$. Interpret $\Gamma_0 = (x_0)$ as the path of length 0 starting at $x_0$.  We say that red \textit{survives} on a path $\gamma \in \Gamma_k$ if, for chase-escape with conversion restricted only to the passage times along $\gamma$, the terminal vertex of $\gamma$ is at some point colored red. We emphasize that survival along $\gamma$ ignores the influence of red and blue from all edges not belonging to $\gamma$. 
 
Let $A_k = A_k(\lambda)$ be the event that $k$ is ever colored red in chase-escape with conversion on the infinite path $0,1,2,\hdots$ with $0$ initially red. Observe that $$\P_\lambda(\text{red survives on a path $\gamma$ of length $k$}) = \P_\lambda(A_k).$$ A necessary condition for $A_k$ is that at each site $i=0,1,\hdots, k-1$, red spreads from $i$ to $i+1$ before $i$ converts to blue. This happens with probability $\lambda/ (\lambda + \alpha)$ independently at each site. Thus,
$$\P_\lambda(A_k) \leq \left(\f{\lambda}{\lambda+\alpha} \right)^k.$$
 
Let $H\subseteq G$ denote the set of sites that are ever colored red, so that $X = |H|$. For any vertex $v \in H$ there must be a path along which red survives with $v$ the terminal point. Hence,
\begin{equation}
\label{3}
X \preceq  \sum^\infty_{k=0} \sum_{\gamma \in \Gamma_k} \ind{\text{red survives on } \gamma}.
\end{equation}
Taking expectation and using the bound on $\P_\lambda(A_k)$ plus the fact that $|\Gamma_0|= 1$ and $|\Gamma_k|\leq d (d-1)^{k-1}$ for $k \geq 1$ gives 
\begin{align}
    \E_\lambda[X] \leq \sum_{k=0}^\infty |\Gamma_k| \P_\lambda(A_k) \leq  {1 +} \f{d \lambda}{\lambda + \alpha} \sum_{k=0}^\infty \left(\f {(d-1)\lambda }{\lambda + \alpha}\right)^{k}  . \label{eq:ER}
\end{align}
This quantity is finite so long as $(d-1)\lambda/(\lambda + \alpha)<1$, which holds for $\lambda < \alpha/(d-2)$. 

Now we prove the upper bound on $\lambda_c(\alpha)$. 
We will use here the notation of defining the process through random variables associated to directed edges which was introduced in Section \ref{sec:tree}. For each vertex $x \in G$, let $\mathcal{N}(x)$ be the collection of vertices connected to $x$. As a reminder of the necessary notation, recall that for any vertex $x \in G$ and $y$ a neighbor of $x$, $y \in \mathcal{N}(x)$, we define $\mathcal{E}_{x,y}$ to be the directed edge connecting $x$ to $y$ and $\mathcal{E}_{y,x}$ as the directed edge connecting $y$ to $x$. As before, we assign the random variable $\vec R_{x,y}$ to the edge $\mathcal{E}_{x,y}$, which defines the time until red spreads from $x$ to $y$ and the random variable $\vec B_{y,x}$ to the edge $\mathcal{E}_{y,x}$ to be the time until blue spreads from $y$ to $x$. Let $C_x$ be the time until vertex $x$ converts from red to blue. With this notation in place, we call a vertex $x \in G$ \emph{good} if $$\max_{y \in \mathcal{N}(x)} \vec R_{x,y} < \min_{y \in \mathcal{N}(x)} \vec B_{y,x} \wedge C_x.$$ 
In words, this means that a vertex $x$ is good if the time it would take to spread red to all its neighboring sites is less than the shortest possible length of time after $x$ becomes red before blue can affect the site, either through conversion or spreading. 

If $\mathcal C$ is the connected component of good sites containing $x_0$, it is easy to deduce that all sites of $\mathcal C$ will at some point be colored red. This is because red spreads to all of its neighbors from each site in $\mathcal C$ before being affected by blue. Thus, {$X \succeq |\mathcal C|$}. We then have $\P_\lambda(X = \infty) \geq \P_\lambda(|\mathcal C| = \infty)$. Recall that $p_c = p_c(G)$ is the critical threshold for Bernoulli site percolation on $G$. It suffices to show that for each $x \in G$, $\P_\lambda(x \text{ is good})>p_c$ for $\lambda$ sufficiently large as this implies that $\P_\lambda(X = \infty) > 0$. 

The probability a given site {$x$ is good is the probability the maximum of $|\mathcal{N}(x)|$ independent $\Exp(\lambda)$ random variables is smaller than the minimum of $|\mathcal{N}(x)|$ independent $\Exp(1)$ random variables and one independent $\Exp(\alpha)$ random variables. Because $|\mathcal{N}(x)| \leq d$, the probability $x$ is good is bounded below by }the probability that $R_d$, the maximum of $d$ independent $\Exp(\lambda)$ random variables, is smaller than $B_d$, the minimum of $d$ independent $\Exp(1)$ random variables and one independent $\Exp(\alpha)$ random variable. Thus,
$$\P_\lambda(x \text{ is good}) {\geq } \P_\lambda(R_d < B_d) = \int_0^\infty \P_\lambda{(R_d <  z) f_{B_d}(z)} dz,$$
{where $f_{B_d}(z)$ is the density of $B_d$.} 

We would like to find a condition on $\lambda$ that ensures this integral is strictly larger than $p_c$. 
One way to do this is to note that $R_d \preceq R'_d \sim \text{Gamma}(d,\lambda)$ i.e. the sum of $d$ independent $\Exp(\lambda)$ random variables, since the maximum of $d$ independent $\Exp(\lambda)$ random variables is bounded by the sum. This has a simpler formulation
$$\P_\lambda(x \text{ is good})\geq \P_\lambda( R_d' < B_d)  = \left(\frac{\lambda}{\lambda + d + \alpha}\right)^d.$$
The formula is easily derived from iteratively applying the memoryless property of the exponential distribution since each of the $d$ independent $\Exp(\lambda)$ random variables in the sum comprising $R_d'$ must occur before the $\Exp(d + \alpha)$-distributed random variable $B_d$. 

Letting $p = p_c$, some algebra then gives that 
$$\lambda > \frac{p^{\frac{1}{d}}(d+\alpha)}{1 - p^{\frac{1}{d}}} \implies \left(\frac{\lambda}{\lambda + d + \alpha} \right)^d > p  .$$ 
Replacing $p^{1/d}$ with $1$ in the numerator on the left yields that whenever $\lambda \geq (d+\alpha)/(1- p^{1/d})$ we have $\P_{\lambda}(X=\infty)>0$. Thus, $\lambda_c(\alpha)$ is no larger than $(d+\alpha)/(1- p^{1/d})$.

 \bibliographystyle{plain}
 \bibliography{paper.bib}

\end{document}

%% file: graphs.tex
\begin{figure}[h!]
\begin{center}
\begin{tikzpicture}[scale=.4, mycircle/.style={draw=black, fill=white, inner sep = .1 cm}]

\def\rad{0.25}

\begin{scope}[xshift=-1.5in]

\draw (-3,0) -- ++(1.5,0)-- ++(1.5,0)-- ++(1.5,0) -- ++(1.5,0) ;
\draw[mycircle,fill=red] (-3,0) circle (\rad);
\foreach \y in {-1.5,0,1.5,3}{
\draw[mycircle] (\y,0) circle (\rad);
}
\end{scope}

\begin{scope}[xshift=0.5in]

\coordinate (B0) at (0,0);
\coordinate (B00) at (1.5,1);
\coordinate (B01) at (1.5,-1);
\coordinate (B000) at (3,1.5);
\coordinate (B001) at (3,0.5);
\coordinate (B010) at (3,-0.5);
\coordinate (B011) at (3,-1.5);
\coordinate (B0000) at (4.5,2.25);
\coordinate (B0001) at (4.5,1.5);
\coordinate (B0010) at (4.5,0.5);
\coordinate (B0100) at (4.5,-0.5);
\coordinate (B0110) at (4.5,-1.5);

\draw 	(B0) -- (B00)
		(B0) -- (B01)
		(B00) -- (B000)
		(B00) -- (B001)
		(B01) -- (B010)
		(B01) -- (B011)
		(B000) -- (B0000)
		(B000) -- (B0001)
		(B001) -- (B0010)
		(B010) -- (B0100)
		(B011) -- (B0110);

\draw[mycircle,fill=red] (B0) circle (\rad);
\foreach \x in {B00,B01,B000,B001,B010,B011,B0000,B0001,B0010, B0100,B0110}{
\draw[mycircle] (\x) circle (\rad);}

\end{scope}

\begin{scope}[xshift=4.2in]

\foreach \x in {1,...,20}{
\coordinate (C\x) at (\x*360/20:2.5);
\draw (0,0) -- (C\x);
\draw[mycircle] (C\x) circle (\rad);
}

\draw[mycircle,fill=red] (0,0) circle (\rad);

\end{scope}

\begin{scope}[xshift=7in]
\def\n{16}

\foreach \x in {1,...,\n}{
	\coordinate (D\x) at (\x*360/16:2.5);
	}

\foreach\x in {1,...,\n}{
	\foreach\y in {1,...,\n}{
		\draw (D\x) -- (D\y);}}

\foreach \x in {1,2,...,7}{
	\draw[mycircle] (D\x) circle (\rad);
 	}
\foreach \x in {9,...,16}{
	\draw[mycircle] (D\x) circle (\rad);
 	}

\draw[mycircle,fill=red] (D8) circle (\rad);

\end{scope}
\end{tikzpicture}
\end{center}
\caption{From left to right: The positive integers $1, \hdots ,5$, a tree with 12 vertices, the star graph $S_{20}$, 
and the complete graph $K_{16}$. The root $x_0$ is shaded red in each example.} \label{fig:graphs}
\end{figure}

%% file: Y0.tex
\begin{figure}
    \centering

    \begin{subfigure}[b]{0.48\textwidth}
        \centering
        \begin{tikzpicture}[scale = .5]
            \node[circle,fill=red, draw = black, inner sep = .1 cm] at (0, 0) {};
            \foreach \n in {1,...,8}{
                \node[circle,fill=red, draw = black, inner sep = .1 cm] at (\n, 0) {};
                \draw (\n -.75,0) -- (\n-.25,0);
            }    
            \foreach \n in {9,...,10}{
                \node[circle,fill=white, draw = black, inner sep = .1 cm] at (\n, 0) {};
                \draw (\n -.75,0) -- (\n-.25,0);
            }    
            \node[circle,fill=blue, draw = black] at (4,0) {};
            \node at (4,-1) {$M$};
            \node at (8,-1) {$N$};
        \end{tikzpicture}

\vspace{.25 cm}
        
        \begin{tikzpicture}[scale = .5]
            \node[circle,fill=red, draw = black, inner sep = .1 cm] at (0, 0) {};
            \foreach \n in {1,...,7}{
                \node[circle,fill=red, draw = black, inner sep = .1 cm] at (\n, 0) {};
                \draw (\n -.75,0) -- (\n-.25,0);
            }    
            \foreach \n in {8,...,10}{
                \node[circle,fill=white, draw = black, inner sep = .1 cm] at (\n, 0) {};
                \draw (\n -.75,0) -- (\n-.25,0);
            }    
            \node[circle,fill=blue, draw = black] at (4,0) {};
            \node[circle,fill=blue, draw = black] at (5,0) {};
            \node[circle,fill=blue, draw = black] at (1,0) {};
            \node[circle,fill=blue, draw = black] at (2,0) {};
            \node at (5,-1) {$M'_{\gamma'}$};
            \node at (7,-1) {$N'_{\gamma'}$};
        \end{tikzpicture}
    \end{subfigure}
    \hfill
    \begin{subfigure}[b]{0.48\textwidth}
        \centering
        \begin{tikzpicture}[scale = .5]
            \node[circle,fill=red, draw = black, inner sep = .1 cm] at (0, 0) {};
            \foreach \n in {1,...,8}{
                \node[circle,fill=red, draw = black, inner sep = .1 cm] at (\n, 0) {};
                \draw (\n -.75,0) -- (\n-.25,0);
            }    
            \foreach \n in {9,...,10}{
                \node[circle,fill=white, draw = black, inner sep = .1 cm] at (\n, 0) {};
                \draw (\n -.75,0) -- (\n-.25,0);
            }    
            \node[circle,fill=blue, draw = black] at (4,0) {};
            \node at (4,-1) {$M$};
            \node at (8,-1) {$N$};
        \end{tikzpicture}

\vspace{.25 cm}
        
        \begin{tikzpicture}[scale = .5]
            \node[circle,fill=red, draw = black, inner sep = .1 cm] at (0, 0) {};
            \foreach \n in {1,...,3}{
                \node[circle,fill=red, draw = black, inner sep = .1 cm] at (\n, 0) {};
                \draw (\n -.75,0) -- (\n-.25,0);
            }    
            \foreach \n in {4,...,10}{
                \node[circle,fill=white, draw = black, inner sep = .1 cm] at (\n, 0) {};
                \draw (\n -.75,0) -- (\n-.25,0);
            }    
            \node[circle,fill=blue, draw = black] at (3,0) {};
            \node[circle,fill=blue, draw = black] at (1,0) {};
            \node at (4,-1) {$M'_{\gamma'}=N'_{\gamma'}$};
        \end{tikzpicture}
    \end{subfigure}
    \caption{Two example realizations of chase-escape with conversion in the coupling to prove (i) and (ii) in \thref{lem:X'}. In each subfigure (left and right), the top image is the process with $\lambda$ and $\alpha$ at time $\gamma$, and the bottom image is the process with $\lambda'$ and $\alpha'$ at time $\gamma'$. On the left, we have $M'_{\gamma'} > M$ and $N'_{\gamma'}< N$. On the right, we have $M'_{\gamma'}=N'_{\gamma'} < N$. In both subfigures, the claimed relations $N'_{\gamma'} \leq N$ and $N'_{\gamma'}- M'_{\gamma'} \leq N - M$ hold. }
    \label{fig:Y0}
\end{figure}